\documentclass[11pt]{article}
\usepackage{graphicx}
\usepackage{amscd, amssymb}
\usepackage{enumerate}
\usepackage{hyperref}
\usepackage{lscape}
\newenvironment{eq}{\begin{equation}}{\end{equation}}
\newenvironment{proof}{{\bf Proof}:}{\vskip 5mm }

\newtheorem{proposition}{Proposition}[subsection]
\newtheorem{lemma}[proposition]{Lemma}
\newtheorem{definition}[proposition]{Definition}

\newtheorem{example}[proposition]{Example}
\newtheorem{remark}[proposition]{Remark}

\newcommand{\llabel}[1]{\label{#1}}
\newcommand{\comment}[1]{}
\newcommand{\sr}{\rightarrow}

\newcommand{\nn}{{\bf N\rm}}

\newcommand{\wt}{\widetilde}

{\vskip
3mm}

\newcommand{\spc}{{\,\,\,\,\,\,\,}}
\newcommand{\impl}{{\Rightarrow}}

\begin{document}
\parskip = 2mm
\begin{center}
{\bf\Large C-system of a module over a monad on sets\footnote{\em 2000 Mathematical Subject Classification: 03B15, 03B22, 03F50, 03G25}}

\vspace{3mm}

{\large\bf Vladimir Voevodsky}\footnote{School of Mathematics, Institute for Advanced Study,
Princeton NJ, USA. e-mail: vladimir@ias.edu}$^,$\footnote{Work on this paper was supported by NSF grant 1100938.}
\vspace {3mm}

{September 2014}  
\end{center}

\begin{abstract}
This is the second paper in a series started in \cite{Csubsystems} which aims to provide mathematical descriptions of objects and constructions related to the first few steps of the semantical theory of dependent type systems. 

We construct for any pair $(R,LM)$, where $R$ is a monad on sets and $LM$ is a left module over $R$, a C-system (``contextual category'') $CC(R,LM)$ and describe, using the results of \cite{Csubsystems}, a class of sub-quotients of $CC(R,LM)$ in terms of objects directly constructed from $R$ and $LM$. In the special case of the monads of expressions associated with  nominal signatures this construction gives the C-systems of general dependent type theories when they are specified by collections of judgements of the four standard kinds.
\end{abstract}


\subsection{Introduction}

The first few steps in all approaches to the semantics of dependent type theories remain insufficiently understood. The constructions which have been worked out in detail in the case of a few particular type systems by dedicated authors are being extended to the wide variety of type systems under consideration today by analogy. This is not acceptable in mathematics. Instead we should be able to obtain the required results for new type systems by {\em specialization} of general theorems formulated and proved for abstract objects the instances of which combine together to produce a given type system. 

One such class of objects is the class of C-systems introduced in \cite{Cartmell0} (see also \cite{Cartmell1}) under the name ``contextual categories''. A modified axiomatics of C-systems and the construction of new C-systems as sub-objects and regular quotients of the existing ones in a way convenient for use in type-theoretic applications are considered in \cite{Csubsystems}.

Modules over monads were introduced in \cite{HM2007} in the context of syntax with binding and substitution. 

In the present paper, after some general comments about monads on $Sets$ and their modules, we construct for any such monad $R$ and a left module $LM$ over $R$ a C-system (contextual category) $CC(R,LM)$.   We describe, using the results of \cite{Csubsystems}, all the C-subsystems of $CC(R,LM)$ in terms of objects directly associated with $R$ and $LM$. 

We then define two additional operations $\sigma$ and $\wt{\sigma}$ on $CC(R,LM)$ and describe the regular congruence relations (see \cite{Csubsystems}) on C-subsystems of $CC(R,LM)$ which are compatible in a certain sense with $\sigma$ and $\wt{\sigma}$.

Of a particular interest is the case of ``syntactic'' pairs where $R(\{x_1,\dots,x_n\})$ and $LM(\{x_1,\dots,x_n\})$ are the sets of expressions of some kind with free variables from $\{x_1,\dots,x_n\}$ modulo an equivalence relation such as $\alpha$-equivalence. 

The simplest class of syntactic pairs where $LM=R$ arises from signatures considered in \cite[p.228]{HM2007}. To any such signature $\Sigma$ one associates a class of expressions with bindings and $R(\{x_1,\dots,x_n\})$ is the set of such expressions with free variables from the set $\{x_1,\dots,x_n\}$ modulo the $\alpha$-equivalence.  

Suppose now that we are given a type theory based on the syntax of expressions specified by $\Sigma$ that is formulated in terms of the four kinds of basic judgements originally introduced by Per Martin-Lof in \cite[p.161]{ML78}.  

Since we are only interested in the $\alpha$-equivalence classes of judgements we may assume that the variables declared in the context are taken from the set of natural numbers such that the first declared variable is $1$, the second is $2$ etc.  Then, the set of judgements of the form 
$$(1:A_1,\dots,n:A_n\vdash A\, type)$$
(in the notation of Martin-Lof ``$A\,type\,(1\in A_1,\dots,n\in A_n)$'') can be identified with the set of judgements of the form 
$$(1:A_1,\dots,n:A_n, n+1:A\rhd)$$
stating that the context $(1:A_1,\dots,n:A_n, n+1:A)$ is well-formed. 

With this identification the type theory is specified by four sets $C,\wt{C},Ceq$ and $\wt{Ceq}$ where 
$$C \subset \coprod_{n\ge 0} LM(\emptyset)\times\dots\times LM(\{1,\dots,n-1\})$$
$$\wt{C}\subset  \coprod_{n\ge 0} LM(\emptyset)\times\dots\times LM(\{1,\dots,n-1\})\times R(\{1,\dots,n\})\times LM(\{1,\dots,n\})$$
$$Ceq \subset \coprod_{n\ge 0} LM(\emptyset)\times\dots\times LM(\{1,\dots,n-1\})\times LM(\{1,\dots,n\})^2$$
$$\wt{Ceq} \subset \coprod_{n\ge 0} LM(\emptyset)\times\dots\times LM(\{1,\dots,n-1\})\times R(\{1,\dots,n\})^2\times LM(\{1,\dots,n\})$$ 
On the other hand we show that any pair $(CC,\sim)$, where $CC$ is a sub-C-system of $CC(R,LM)$ and $\sim$ is a regular congruence relation on $CC$, defines four subsets of such form. Proposition \ref{2014.07.10.prop1} spells out the necessary and sufficient conditions that the sets $C,\wt{C},Ceq,\wt{Ceq}$ should satisfy in order to correspond to a pair $(CC,\sim)$. 

A wider class of syntactic pairs $(R,LM)$ that arises from nominal signatures is considered in Section \ref{2014.07.22.sec}.  

This is one the papers extending the material which I started to work on in \cite{NTS}. I would like to thank the Institute Henri Poincare in Paris and the organizers of the ``Proofs'' trimester for their hospitality during the preparation of this paper. The work on this paper was facilitated by discussions with Richard Garner and Egbert Rijke. 

Notation: For morphisms $f:X\sr Y$ and $g:Y\sr Z$ we denote their composition as $f\circ g$. For functors $F:{\cal C}\sr {\cal C}'$, $G:{\cal C}'\sr {\cal C}''$ we use the standard notation $G\circ F$ for their composition.

\subsection{Left modules over monads}
Recall (cf. \cite{HM2007}) that a monad on a category $\cal C$ is a functor $M:{\cal C}\sr {\cal C}$ together with two families of morphisms:
\begin{enumerate}
\item for any $X\in {\cal C}$, a morphism $\eta_X:X \sr R(X)$,
\item for any $X\in {\cal C}$, a morphism $\mu_X:R(R(X))\sr R(X)$
\end{enumerate}
which satisfy certain conditions. For objects $X$, $X'$ and a morphism $f:X'\sr R(X)$, the composition $R(X') \stackrel{R(f)}{\sr} R(R(X))\stackrel{\mu_X}{\sr} R(X)$ is a morphism  $bind(f):R(X')\sr R(X)$. This allows one to describe monads as follows:
\begin{lemma}
\llabel{2014.06.30.l1}
The construction outline above defines an equivalence between (the type of)  monads  on $\cal C$ and (the type of) collections of data of the form:
\begin{enumerate}
\item for every object $X$ an object $R(X)$,
\item for every object $X$ a morphism $\eta_X: X \sr R(X)$,
\item for every two objects $X$, $X'$ and a morphism $f:X\sr R(X')$, a morphism $ bind(f):R(X)\sr R(X')$
\end{enumerate}
which satisfy the following conditions:
\begin{enumerate}
\item for an object $X$, $ bind(\eta_X)=id_{R(X)}$,
\item for a morphism $f:X\sr X'$, $\eta_X\circ bind(f)=f$,
\item for two morphisms $f:X\sr R(X')$, $g:X'\sr R(X'')$, $ bind(f \circ bind(g))= bind(f)\circ bind(g)$.
\end{enumerate}
\end{lemma}
\begin{proof}
Straightforward. Cf. \cite{Moggi91}, \cite[Prop. 1]{HM2010}.
\end{proof}
\begin{lemma}
\llabel{2014.07.28.l2}
Let $R$ be a monad on the product category ${\cal C}\times {\cal D}$. Let $A\in {\cal D}$. Then the functor $R_{A,1}:X\mapsto pr_{\cal C}(R(X,A))$ has a natural structure of a monad on $\cal C$.
\end{lemma}
\begin{proof}
One defines the morphisms $\eta_X:X\sr R_{A,1}(X)$ by
$$\eta_X := pr_{\cal C}(\eta_{(X,A)})$$
and morphisms $ bind(f):R_{A,1}(X)\sr R_{A,1}(X')$ for $f:X\sr R_{A,1}(X')$ by
$$ bind(f) := pr_{\cal C}( bind(f,pr_{\cal D}(\eta_{(X,A)})))$$
The verification of the conditions of Lemma \ref{2014.06.30.l1} is straightforward. 
\end{proof}
The concept of a module over a monad was first explicitly introduced in \cite{HM2007}.
\begin{definition}
\llabel{2014.07.26.d1} Let $R$ be a monad on a category ${\cal C}$. A left module over $R$ with values in a category $\cal D$ is a functor $LM:{\cal C}\sr {\cal D}$ together with, for all $X, X'\in {\cal C}$ and $f:X\sr R(X')$, a morphism $\rho(f):LM(X)\sr LM(X')$ such that
\begin{enumerate}
\item $\rho(\eta_X)=Id_{LM(X)}$,
\item for $f:X\sr R(X')$, $g:X'\sr R(X'')$, $\rho(f) \rho(g)=\rho(f\, bind(g))$.
\end{enumerate}
\end{definition}
One verifies easily (cf. \cite[Def. 9]{HM2010}) that a left $R$-module structure on $LM$ is the same as a natural transformation $LM\circ R\sr LM$ which satisfies the expected compatibility conditions with respect to $Id\sr R$ and $R\circ R\sr R$. 
\begin{lemma}
\llabel{2014.07.28.l1}
Let $R$ be a monad on a category $\cal C$. Then one has:
\begin{enumerate}
\item If $LM_1$, $LM_2$ are left $R$-modules with values in ${\cal D}_1$ and ${\cal D}_2$ respectively then the functor $X\mapsto (LM_1(X),LM_2(X))$ has a natural structure of a left $R$-module with values in ${\cal D}_1\times {\cal D}_2$.
\item If $LM$ is a left $R$-module with values in $\cal D$ and $F:{\cal D}\sr {\cal D'}$ is a functor then $F\circ LM$ has a natural structure of a left $R$-module with values in $\cal D'$.
\end{enumerate}
\end{lemma}
\begin{proof}
Straightforward.
\end{proof}
\begin{lemma}
\llabel{2014.07.28.l3}
Under the assumptions and in the notation of Lemma \ref{2014.07.28.l2} the morphisms
$$\rho(f:X\sr R_{A,1}(X'))= bind(f, pr_{\cal D}(\eta_{(X',A)})) : R(X,A)\sr R(X',A)$$
define a structure of a left $R_{A,1}$-module with values in ${\cal C}\times {\cal D}$ on the functor 
$$M_{A,1}:X\mapsto R(X,A)$$
\end{lemma}
\begin{proof}
Direct verification of the conditions of Definition \ref{2014.07.26.d1}.
\end{proof}

In the case of a monad $R$ on $Sets$ and a left $R$-module $LM$ with values in $Sets$, for $E\in LM(\{x_1,\dots,x_n\})$ and $f:\{x_1,\dots,x_n\}\sr R(X')$ such that $f(x_i)=f_i$ we write $\rho(f)(E)$ as $E(f_1/x_1,\dots,f_n/x_n)$. 

For $E\in LM(\{1,\dots,m\})$ and $n\ge 1$ we set:
$$t_n(E)=E[n+1/n,n+2/n+1,\dots,m+1/m]$$
$$s_n(E)=E[n/n+1,n+1/n+2,\dots,m-1/m]$$
If we were numbering elements of a set with $n$ elements from $0$ then we would have $t_n=LM(\partial_{n-1})$ and $s_n=LM(\sigma_{n-1})$ where $\partial_i$ and $\sigma_i$ are the usual generators of the simplicial category. 

For a monad $R$ on $Sets$ we let $R-cor$ (``R-correspondences'') to be the full subcategory of the Kleisli category of $R$ whose objects are finite sets. Recall, that the set of morphisms from $X$ to $Y$ in $R-cor$ is the set of maps from $X$ to $R(Y)$ i.e. $R(Y)^X$ (in other words, $R-cor$ is the 
category of free, finitely generated $R$-algebras).   

We further let $C(R)$ denote the pre-category\footnote{See the introduction to \cite{Csubsystems}.}  with 
$$Ob(C(R))=\nn$$
$$Mor(C(R))=\coprod_{m,n\in\nn} R(\{1,\dots,m\})^n$$
which is equivalent, as a category, to $(R-cor)^{op}$.
\begin{remark}\rm
A finitary monad (on sets) is a monad $R:Sets\sr Sets$ that, as a functor, commutes with filtering colimits. Since any set is, canonically, the colimit of the filtering diagram of its finite subsets, a functor $Sets \sr Sets$ that commutes with filtering colimits can be equivalently described as a functor $FSets \sr Sets$ where $FSets$ is the category of finite sets. Furthermore, Lemma \ref{2014.06.30.l1} can be used to show that finitary monads on $Sets$ can be defined as collections of data of the form:
\begin{enumerate}
\item for every finite set $X$ a set $R(X)$,
\item for every finite set $X$ a function $\eta_X: X \sr R(X)$,
\item for every finite sets $X$, $X'$ and a function $f:X\sr R(X')$, a function $ bind(f):R(X)\sr R(X')$
\end{enumerate}
which satisfy the conditions:
\begin{enumerate}
\item for a finite set $X$, $ bind(\eta_X)=id_{R(X)}$,
\item for a function of finite sets $f:X\sr X'$, $\eta_X\circ bind(f)=f$,
\item for two functions $f:X\sr R(X')$, $g:X'\sr R(X'')$, $ bind(f\circ bind(g))= bind(f)\circ bind(g)$.
\end{enumerate}
This description shows that for any monad $R$ the restriction of $R$ to a functor $R^{fin}:FSets\sr Sets$ is a finitary monad. 

Similar observations apply to left $R$-modules. The constructions of this paper, while done for a general pair $(R,LM)$, only depend on the corresponding finitary pair $(R^{fin},LM^{fin})$. 
\end{remark}
\begin{remark}\rm
The correspondence $R\mapsto C(R)$ defines an equivalence between the type of the finitary monads on $Sets$ and the type of the pre-category structures on $\nn$ that extend the pre-category structure of finite sets and where the addition remains to be the coproduct. 
\end{remark}
\begin{remark}\rm A finitary sub-monad of $R$ is the same as a sub-pre-category in $C(R)$ which contains all objects. Intersection of two sub-monads is a sub-monad which allows one to speak of sub-monads generated by a set of elements. 
\end{remark}

\subsection{The C-system $CC(R,LM)$.}
Let $R$ be a monad on $Sets$ and $LM$ a left module over $R$ with values in $Sets$. Let $CC(R,LM)$ be the pre-category whose set of objects is $Ob(CC(R,LM))=\amalg_{n\ge 0} Ob_n$ where 
$$Ob_n=LM(\emptyset)\times\dots\times LM(\{1,\dots,n-1\})$$
and the set of morphisms is
$$Mor(CC(R,LM))=\coprod_{m,n\ge 0} Ob_m\times Ob_n\times R(\{1,\dots,m\})^n$$
with the obvious domain and codomain maps. The composition of morphisms is defined in the same way as in $C(R)$ such that the mapping $Ob(CC(R,LM))\sr \nn$ which sends all elements of $Ob_n$ to $n$, is a functor from $CC(R,LM)$ to $C(R)$. The associativity of compositions follows immediately from the associativity of compositions in $R-cor$. 

Note that if $LM(\emptyset)=\emptyset$ then $CC(R,LM)=\emptyset$ and otherwise the functor $CC(R,LM)\sr C(R)$ is an equivalence, so that in the second case $C(R)$ and $CC(R,LM)$ are indistinguishable as categories. However, as pre-categories they are quite different unless $LM=(X\mapsto pt)$ in which case the functor $CC(R,LM)\sr C(R)$ is an isomorphism. 

The pre-category $CC(R,LM)$ is given the structure of a C-system as follows. The final object is the only element of $Ob_0$, the map $ft$ is defined by the rule
$$ft(T_1,\dots,T_n)=(T_1,\dots,T_{n-1}).$$
The canonical pull-back square defined by an object $(T_1,\dots,T_{n+1})$ and a morphism 
$$(f_1,\dots,f_{n}):(R_1,\dots,R_m)\sr (T_1,\dots,T_{n})$$
is of the form:
\begin{eq}
\label{2009.11.05.oldeq1}
\begin{CD}
(R_1,\dots,R_m,T_{n+1}(f_1/1,\dots, f_{n}/n)) @>(f_1,\dots,f_{n},m+1)>> (T_1,\dots,T_{n+1})\\
@V(1,\dots,m)VV @VV(1,\dots,n)V\\
(R_1,\dots,R_m) @>(f_1,\dots,f_{n})>> (T_1,\dots,T_{n})
\end{CD}
\end{eq}
\begin{proposition}
\llabel{2009.10.01.prop2}
With the structure defined above $CC(R,LM)$ is a C-system.
\end{proposition}
\begin{proof}
Straightforward.
\end{proof}
\begin{remark}\rm
There is another construction of a pre-category from $(R,LM)$ which takes as an additional parameter a set $Var$ which is called the set of variables. Let $F_n(Var)$ be the set of sequences of length $n$ of pair-wise distinct elements of $Var$. Define the pre-category $CC(R,LM,Var)$ as follows. The set of objects of $CC(R,LM,Var)$ is 
$$Ob(CC(R,LM,Var))= \amalg_{n\ge 0} \amalg_{(x_1,\dots,x_n)\in F_n(Var)} LM(\emptyset)\times\dots\times LM(\{x_1,\dots,x_{n-1}\})$$
For compatibility with the traditional type theory we will write the elements of $Ob(CC(R,LM,X))$ as sequences of the form $x_1:E_1,\dots,x_n:E_n$. The set of morphisms is given by
$$Hom_{CC(R,LM,Var)}((x_1:E_1,\dots,x_m:E_m),(y_1:T_1,\dots,y_n:T_n))=R(\{x_1,\dots,x_m\})^n$$
The composition is defined in such a way that the projection 
$$(x_1:E_1,\dots,x_n:E_n)\mapsto (E_1,E_2(1/x_1),\dots,E_n(1/x_1,\dots,n-1/x_{n-1}))$$
is a functor from $CC(R,LM,Var)$ to $CC(R,LM)$. 

This functor is clearly an equivalence of categories but not an isomorphism of pre-categories. 

There are an obvious final object and the map $ft$ on $CC(R,LM,Var)$. 

There is however a real problem in making it into a C-system  which is due to the following. Consider an object $(y_1:T_1,\dots,y_{n+1}:T_{n+1})$ and a morphism $(f_1,\dots,f_n):(x_1:R_1,\dots,x_m:R_m)\sr (y_1:T_1,\dots,y_{n}:T_{n})$. In order for the functor to $CC(R,LM)$ to be a C-system morphism the canonical square build on this pair should have the form
$$
\begin{CD}
(x_1:R_1,\dots,x_m:R_m,x_{m+1}:T_{n+1}(f_1/1,\dots, f_{n}/n)) @>(f_1,\dots,f_{n},x_{n+1})>> (y_1:T_1,\dots,y_{n+1}:T_{n+1})\\
@VVV @VVV\\
(x_1:R_1,\dots,x_m:R_m) @>(f_1,\dots,f_{n})>> (y_1:T_1,\dots,y_n:T_{n})
\end{CD}
$$
where $x_{m+1}$ is an element of $Var$ which is distinct from each of the elements $x_1,\dots,x_m$. Moreover, we should choose  $x_{m+1}$ in such a way the the resulting construction satisfies the C-system axioms for $(f_1,\dots,f_{n})=Id$ and for the compositions $(g_1,\dots,g_m)\circ (f_1,\dots,f_n)$. One can easily see that no such choice is possible for a finite set $Var$. At the moment it is not clear to me whether or not it is possible for an infinite $Var$.
\end{remark}

Recall from \cite{Csubsystems} that for a C-system $CC$ one defines $\wt{Ob}(CC)$ as the subset of $Mor(CC)$ which consists of morphisms $s$ of the form $ft(X)\sr X$ such that $l(X)>0$ and $s\circ p_X=Id_{ft(X)}$. 
\begin{lemma}
\llabel{2014.06.30.l2}
One has:
$$\wt{Ob}(CC(R,LM))\cong \coprod_{n\ge 0} LM(\emptyset)\times\dots\times LM(\{1,\dots,n\})\times R(\{1,\dots,n\})$$
\end{lemma}
\begin{proof}
An element of $\wt{Ob}(CC(R,LM))$ is a section $s$ of the canonical morphism $p_{\Gamma}:\Gamma\sr ft(\Gamma)$. It follows immediately from the definition of $CC(R,LM)$ that for $\Gamma=(E_1,\dots,E_{n+1})$, a morphism $(f_1,\dots,f_{n+1})\in R(\{1,\dots,n\})^{n+1}$ from $ft(\Gamma)$ to $\Gamma$ is a section of $p_{\Gamma}$ if an only if $f_i=i$ for $i=1,\dots,n$. Therefore, any such section is determined by its last component $f_{n+1}$ and mapping
$((E_1,\dots,E_n), (E_1,\dots,E_{n+1}), (f_1,\dots,f_{n+1}))$ to $(E_1,\dots,E_n,E_{n+1},f_{n+1})$ we get a bijection
\begin{eq}
\llabel{2009.10.15.eq2}
\wt{Ob}(CC(R,LM))\cong \coprod_{n\ge 0} LM(\emptyset)\times\dots\times LM(\{1,\dots,n\})\times R(\{1,\dots,n\})
\end{eq}
\end{proof}
Using the notations of type theory we can write elements of $Ob(CC(R,LM))$ as 
$$\Gamma=(T_1,\dots,T_n\rhd)$$
where $T_i\in LM(\{1,\dots,i-1\})$ and the elements of $\wt{Ob}(CC(R,LM))$ as 
$${\cal J} = (T_1,\dots,T_n\vdash t:T)$$
where $T_i\in LM(\{1,\dots,i-1\})$, $T\in LM(\{1,\dots,n\})$ and $t\in R(\{1,\dots,n\})$.

In this notation the operations $T,\wt{T},S,\wt{S}$ and $\delta$ which were introduced in \cite{Csubsystems} take the form:
\begin{enumerate}
\item $T((\Gamma,T_{n+1}\rhd),(\Gamma,\Delta\rhd))=(\Gamma,T_{n+1},t_{n+1}(\Delta)\rhd)$ when $l(\Gamma)=n$,
\item $\wt{T}((\Gamma,T_{n+1}\rhd),(\Gamma,\Delta\vdash r:R))=(\Gamma,T_{n+1},t_{n+1}(\Delta)\vdash t_{n+1}(r:R))$ when $l(\Gamma)=n$,
\item $S((\Gamma\vdash s:S),(\Gamma,S,\Delta\rhd))=(\Gamma,s_{n+1}(\Delta[s/n+1])\rhd)$ when $l(\Gamma)=n$,
\item $\wt{S}((\Gamma\vdash s:S),(\Gamma,S,\Delta\vdash r:R))=(\Gamma,s_{n+1}(\Delta[s/n+1])\vdash s_{n+1}((r:R)[s/n+1])$ when $l(\Gamma)=n$,
\item $\delta(\Gamma,T\rhd)=(\Gamma,T\vdash (n+1):T)$ when $l(\Gamma)=n$.
\end{enumerate}

\begin{remark}\rm
\llabel{2014.09.28.rm1}
One can easily construct on the function $(R,LM)\mapsto CC(R,LM)$ the structure of a functor from the ``large module category'' of \cite{HM2008} to the category of C-systems and their homomorphisms.
\end{remark}

\subsection{C-subsystems of $CC(R,LM)$.}
Let $CC$ be a C-subsystem of $CC(R,LM)$.  By \cite{Csubsystems} $CC$ is determined by the subsets $C=Ob(CC)$ and $\wt{C}=\wt{Ob}(CC)$ in $Ob(CC(R,LM))$ and $\wt{Ob}(CC(R,LM))$. 

For $\Gamma=(E_1,\dots,E_n)$ we write $(\Gamma\rhd_{C})$ if $(E_1,\dots,E_n)$ is in $C$ and $(\Gamma\vdash_{\wt{C}} t:T)$  if  $(E_1,\dots,E_n,T,t)$ is in $\wt{C}$. 

The following result is an immediate corollary of \cite[Proposition 4.3]{Csubsystems} together with the description of the operations $T,\wt{T},S,\wt{S}$ and $\delta$ for $CC(R,LM)$ which is given above. 
\begin{proposition}
\llabel{2009.10.16.prop3}
Let $(R,LM)$ be a monad on $Sets$ and a left module over it with values in $Sets$.  A pair of subsets 
$$C\subset \coprod_{n\ge 0} \prod_{i=0}^{n-1} LM(\{1,\dots,i\})$$
$$\wt{C}\subset \coprod_{n\ge 0}  (\prod_{i=0}^{n} LM(\{1,\dots,i\}))\times R(\{1,\dots,n\})$$
corresponds to a C-subsystem $CC$ of $CC(R,LM)$  if and only if the following conditions hold:
\begin{enumerate}
\item $(\rhd_{C})$
\item $(\Gamma, T\rhd_{C})\Rightarrow (\Gamma\rhd_{C})$
\item $(\Gamma\vdash_{\wt{C}} r:R)\Rightarrow (\Gamma,R\rhd_{C})$
\item $(\Gamma, T\rhd_{C})\wedge(\Gamma,\Delta,\vdash_{\wt{C}} r:R)\Rightarrow  (\Gamma, T, t_{n+1}(\Delta)\vdash_{\wt{C}} t_{n+1} (r: R))$
where $n=l(\Gamma_1)$
\item  $(\Gamma\vdash_{\wt{C}}  s:S)\wedge (\Gamma,S,\Delta\vdash_{\wt{C}} r:R)\Rightarrow (\Gamma, s_{n+1}(\Delta[s/n+1]) \vdash_{\wt{C}} s_{n+1} (( r : R ) [s/n+1]))$ where $n=l(\Gamma_1)$,
\item $(\Gamma,T\rhd_{C})\Rightarrow (\Gamma,T\vdash_{\wt{C}} n+1:T)$ where $n=l(\Gamma)$.
\end{enumerate}
\end{proposition}
Note that conditions (4) and (5) together with condition (6) and condition (3) imply the following 
\begin{description}
\item[{\em 4a}] $(\Gamma, T\rhd_{C})\wedge (\Gamma,\Delta\rhd_{C})\Rightarrow (\Gamma, T, t_{n+1}(\Delta)\rhd_{C})$ where $n=l(\Gamma_1)$,
\item[{\em 5a}] $(\Gamma\vdash_{\wt{C}}  s:S)\wedge (\Gamma,S,\Delta\rhd_{C})\Rightarrow (\Gamma, s_{n+1}(\Delta[s/n+1])\rhd_{C})$ where $n=l(\Gamma_1)$.
\end{description}
Note also that modulo condition (2), condition (1) is equivalent to the condition that $C\ne\emptyset$. 

\begin{remark}\rm\llabel{2010.08.07.rem1} If one re-writes the conditions of Proposition \ref{2009.10.16.prop3} in the more familiar in type theory form where the variables introduced in the context are named rather than directly numbered one arrives at the following rules:

\begin{center}

$$\frac{}{\rhd_{C}}\,\,\,\,\,\,\,\,\,\,
\frac{x_1:T_1,\dots,x_n:T_n\rhd_{C}}{x_1:T_1,\dots,x_{n-1}:T_{n-1}\rhd_{C}} \,\,\,\,\,\,\,\,\,\, 
\frac{x_1:T_1,\dots,x_n:T_n\vdash_{\wt{C}} t:T}{x_1:T_1,\dots,x_n:T_n, y:T\rhd_{C}}$$

$$\frac{x_1:T_1,\dots,x_n:T_n, y:T\rhd_{C}\,\,\,\,\,\,\,x_1:T_1,\dots,x_n:T_n,\dots, x_m:T_m\vdash_{\wt{C}} r:R}{x_1:T_1,\dots,x_n:T_n, y:T, x_{n+1}:T_{n+1},\dots,x_m:T_m\vdash_{\wt{C}} r:R}$$

$$\frac{x_1:T_1,\dots,x_n:T_n\vdash_{\wt{C}} s:S\,\,\,\,\,\,\,x_1:T_1,\dots,x_n:T_n,y:S,x_{n+1}:T_{n+1},\dots,x_m:T_m\vdash_{\wt{C}} r:R}
{x_1:T_1,\dots,x_n:T_n,x_{n+1}:T_{n+1}[s/y],\dots,x_m:T_m[s/y]\vdash_{\wt{C}} (r:R)[s/y]}$$

$$\frac{x_1:E_1,\dots,x_n:E_n\rhd_{C}}{x_1:E_1,\dots,x_n:E_n\vdash_{\wt{C}} x_n:E_n}$$

\end{center}
which are similar (and probably equivalent) to the ``basic rules of DTT'' given in \cite[p.585]{Jacobs1}. The advantage of the rules given here is that they are precisely the ones which are necessary and sufficient for a given collection of contexts and judgements to define a C-system.

\end{remark}

\begin{lemma}
\llabel{2009.11.05.l1}
Let $CC$ be as above and let $(E_1,\dots, E_m), (T_1,\dots,T_n)\in Ob(CC)$ and $(f_1,\dots,f_n)\in R(\{1,\dots,m\})^n$. Then  
$$(f_1,\dots,f_n)\in Hom_{CC}((E_1,\dots, E_m), (T_1,\dots,T_n))$$
if and only if $(f_1,\dots,f_{n-1})\in Hom_{CC}((E_1,\dots, E_m), (T_1,\dots,T_{n-1}))$ and 
$$E_1,\dots,E_m\vdash_{\wt{C}} f_n : T_{n}(f_1/1,\dots,f_{n-1}/n-1)$$
\end{lemma}
\begin{proof}
Straightforward using the fact that the canonical pull-back squares in $CC(R,LM)$ are given by (\ref{2009.11.05.oldeq1}).
\end{proof}
\begin{example}\rm
The category $CC(R,R)$ for the identity monad is empty. For the monad of the form $R(X)=pt$ the C-system $CC(R,R)$ has only two subsystems - itself and the trivial one for which $C={pt}$. 

The first non-trivial example is the monad $R(X)=X\amalg \{*\}$. We conjecture that in this case the set of all subsystems of $CC(R,R)$ is {\em uncountable}.

One can probably show this as follows. Let $\epsilon:\nn\sr\{0,1\}$, be a sequence of $0$'s and $1$'s. Consider the C-subsystem of $CC_{\epsilon}$ of $CC(R,R)$ which is generated by the set of elements of the form $(*, 1, 2, \dots, n\rhd)\in Ob(CC(R,R))$ for all $n\ge 0$ and elements $(*,1,\dots,n+1\vdash n+2:*)\in \wt{Ob}(CC(R,R))$ for $n$ such that $\epsilon(n)=1$. 

It should be possible to show that $CC_{\epsilon}\ne CC_{\epsilon'}$ for $\epsilon\ne \epsilon'$ which would imply the conjecture. 
\end{example}

\subsection{Operations $\sigma$ and $\wt{\sigma}$ on $CC(R,LM)$.}
C-systems of the form $CC(R,LM)$ have an important additional structure which will play a role in the next section. This structure is given by two operations:
\begin{enumerate}
\item for $\Gamma=(T_1,\dots,T_n,\dots,T_{n+i})$ and $\Gamma'=(T_1',\dots,T'_{n})$ we set
$$\sigma(\Gamma,\Gamma')=(T_1',\dots,T'_n,T_{n+1},\dots,T_{n+i})$$
This gives us an operation with values in $Ob$ defined on the subset of $Ob\times Ob$ which consists of pairs $(\Gamma,\Gamma')$ such that $l(\Gamma)>l(\Gamma')$,
\item for ${\cal J}=(T_1,\dots,T_{n-1},\dots,T_{n-1+i}\vdash t:T_{n+i})$, $\Gamma'=(T_1',\dots,T_n')$ we set
$$\wt{\sigma}({\cal J},\Gamma')=
\left\{ 
\begin{array}{ll}
(T_1',\dots,T_n',T_{n+1},\dots,T_{n+i-1}:t:T_{n+i})&\mbox{\rm for $i>0$}\\
(T_1',\dots,T_{n-1}'\vdash t:T_n')&\mbox{\rm for $i=0$}
\end{array}
\right.
$$
This gives us an operation with values in $\wt{Ob}$ defined on the subset of $\wt{Ob}\times Ob$ which consists of pairs $({\cal J},\Gamma')$ such that $l(\partial({\cal J}))\le l(\Gamma')$.
\end{enumerate}

\subsection{Regular sub-quotients of $CC(R,LM)$.} 

Let $(R,LM)$ be as above and
$$Ceq\subset \coprod_{n\ge 0}  (\prod_{i=0}^{n-1} LM(\{1,\dots,i\}))\times LM(\{1,\dots,n\})^2$$
$$\wt{Ceq}\subset \coprod_{n\ge 0}  (\prod_{i=0}^{n} LM(\{1,\dots,i\}))\times R(\{1,\dots,n\})^2$$
be two subsets.  

For $\Gamma=(T_1,\dots,T_n)\in ob(CC(R,LM))$ and $S_1,S_2\in LM(\{1,\dots,n\})$ we write $(\Gamma\vdash_{Ceq} S_1=S_2)$ to signify that $(T_1,\dots,T_n,S_1,S_2)\in Ceq$. Similarly for $T\in LM(\{1,\dots,n\})$ and $o,o'\in R(\{1,\dots,n\})$ we write $(\Gamma\vdash_{\wt{Ceq}} o=o':S)$ to signify that $(T_1,\dots,T_n,S,o,o')\in \wt{Ceq}$.  When no confusion is possible we will omit the subscripts $Ceq$ and $\wt{Ceq}$ at $\vdash$. 

Similarly we will write $\rhd$ instead of $\rhd_C$ and $\vdash$ instead of $\vdash_{\wt{C}}$ if the subsets $C$ and $\wt{C}$ are unambiguously  determined by the context.  

\begin{definition}
\llabel{simandsimeq}
Given subsets $C$, $\wt{C}$, $Ceq$, $\wt{Ceq}$ as above define relations $\sim$ on $C$ and $\simeq$ on $\wt{C}$ as follows:
\begin{enumerate}
\item for $\Gamma=(T_1,\dots,T_n)$, $\Gamma'=(T_1',\dots,T_n')$ in $C$ we set  $\Gamma\sim\Gamma'$ iff $ft(\Gamma)\sim ft(\Gamma')$ and 
$$T_1,\dots,T_{n-1}\vdash T_n=T_n',$$
\item for $(\Gamma\vdash o:S)$, $(\Gamma'\vdash o':S')$ in $\wt{C}$ we set $(\Gamma\vdash o:S)\simeq(\Gamma'\vdash o':S')$ iff $(\Gamma,S)\sim(\Gamma',S')$ and 
$$(\Gamma\vdash o=o':S).$$
\end{enumerate}
\end{definition}

\begin{proposition}
\llabel{2014.07.10.prop1}
Let $C$, $\wt{C}$, $Ceq$, $\wt{Ceq}$ be as above and suppose in addition that one has:
\begin{enumerate}
\item $C$ and $\wt{C}$ satisfy conditions (1)-(6) of Proposition \ref{2009.10.16.prop3} which are referred to below as conditions (1.1)-(1.6) of the present proposition,
\item 
$$
\begin{array}{l}
(a)\spc(\Gamma\vdash T=T')\impl (\Gamma,T\rhd)\\
(b)\spc(\Gamma,T\rhd)\impl (\Gamma\vdash T=T)\\
(c )\spc(\Gamma\vdash T=T')\impl(\Gamma\vdash T'=T)\\
(d)\spc(\Gamma\vdash T=T')\wedge(\Gamma\vdash T'=T'')\impl(\Gamma\vdash T=T'')
\end{array}
$$
\item 
$$
\begin{array}{l}
(a)\spc(\Gamma\vdash o=o':T)\impl (\Gamma\vdash o:T)\\
(b)\spc(\Gamma\vdash o:T)\impl (\Gamma\vdash o=o:T)\\
(c )\spc(\Gamma\vdash o=o':T)\impl(\Gamma\vdash o'=o:T)\\
(d)\spc (\Gamma\vdash o=o':T)\wedge(\Gamma\vdash o'=o'':T)\impl(\Gamma\vdash o=o'':T)
\end{array}
$$
\item 
$$
\begin{array}{l}
(a)\spc (\Gamma_1\vdash T=T')\wedge(\Gamma_1,T,\Gamma_2\vdash S=S')\impl(\Gamma_1,T',\Gamma_2\vdash S=S')\\
(b)\spc (\Gamma_1\vdash T=T')\wedge(\Gamma_1,T,\Gamma_2\vdash o=o':S)\impl(\Gamma_1,T',\Gamma_2'\vdash o=o':S)\\
(c )\spc (\Gamma\vdash S=S')\wedge(\Gamma\vdash o=o':S)\impl(\Gamma\vdash o=o':S')
\end{array}
$$
\item 
$$
\begin{array}{ll}
(a)\spc (\Gamma_1,T\rhd)\wedge(\Gamma_1,\Gamma_2\vdash S=S')\impl(\Gamma_1,T,t_{i+1}\Gamma_2\vdash t_{i+1}S=t_{i+1}S')& i=l(\Gamma)\\
(b)\spc (\Gamma_1,T\rhd)\wedge(\Gamma_1,\Gamma_2\vdash o=o':S)\impl(\Gamma_1,T,t_{i+1}\Gamma_2\vdash t_{i+1}o=t_{i+1}o':t_{i+1}S)& i=l(\Gamma)
\end{array}
$$
\item
$$
\begin{array}{ll}
(a)\spc (\Gamma_1,T,\Gamma_2\vdash S=S')\wedge(\Gamma_1\vdash r:T)\impl&\\
(\Gamma_1,s_{i+1}(\Gamma_2[r/i+1])\vdash s_{i+1}(S[r/i+1])=s_{i+1}(S'[r/i+1]))&i=l(\Gamma_1)\\
(b)\spc (\Gamma_1,T,\Gamma_2\vdash o=o':S)\wedge(\Gamma_1\vdash r:T)\impl&\\
(\Gamma_1,s_{i+1}(\Gamma_2[r/i+1])\vdash s_{i+1}(o[r/i+1])=s_{i+1}(o'[r/i+1]):s_{i+1}(S[r/i+1]))&i=l(\Gamma_1)
\end{array}
$$
\item 
$$
\begin{array}{ll}
(a)\spc (\Gamma_1,T,\Gamma_2,S\rhd)\wedge(\Gamma_1\vdash r=r':T)\impl&\\
(\Gamma_1,s_{i+1}(\Gamma_2[r/i+1])\vdash s_{i+1}(S[r/i+1])=s_{i+1}(S[r'/i+1]))&i=l(\Gamma_1)\\
(b)\spc (\Gamma_1,T,\Gamma_2\vdash o:S)\wedge(\Gamma_1\vdash r=r':T)\impl&\\
(\Gamma_1,s_{i+1}(\Gamma_2[r/i+1])\vdash s_{i+1}(o[r/i+1])=s_{i+1}(o[r'/i+1]):s_{i+1}(S[r/i+1]))&i=l(\Gamma_1)
\end{array}
$$
\end{enumerate}
Then the relations $\sim$ and $\simeq$ are equivalence relations on $C$ and $\wt{C}$ which satisfy the conditions of \cite[Proposition 5.4]{Csubsystems} and therefore they correspond to a regular congruence relation on the C-system defined by $(C,\wt{C})$. 
\end{proposition}
\begin{lemma}
\llabel{iseqrelsiml1}
One has:
\begin{enumerate}
\item If conditions (1.2), (4a) of the proposition hold then $(\Gamma\vdash S=S')\wedge(\Gamma\sim\Gamma')\impl (\Gamma'\vdash S=S')$.
\item If conditions (1.2), (1.3), (4a), (4b), (4c) hold then $(\Gamma\vdash o=o':S)\wedge((\Gamma,S)\sim(\Gamma',S'))\impl (\Gamma'\vdash o=o':S')$.
\end{enumerate}
\end{lemma}
\begin{proof}
By induction on $n=l(\Gamma)=l(\Gamma')$.

(1) For $n=0$ the assertion is obvious. Therefore by induction we may assume that $(\Gamma\vdash S=S')\wedge(\Gamma\sim\Gamma')\impl (\Gamma'\vdash S=S')$ for all $i<n$ and all appropriate $\Gamma$,$\Gamma'$, $S$ and $S'$ and that $(T_1,\dots,T_n\vdash S=S')\wedge(T_1,\dots,T_n\sim T'_1,\dots,T'_n)$ holds and we need to show that $(T'_1,\dots,T'_n\vdash S=S')$ holds. Let us show by induction on $j$ that $(T'_1,\dots,T'_j,T_{j+1},\dots,T_n\vdash S=S')$ for all $j=0,\dots,n$. For $j=0$ it is a part of our assumptions. By induction we may assume that $(T'_1,\dots,T'_j,T_{j+1},\dots,T_n\vdash S=S')$. By definition of $\sim$ we have $(T_1,\dots,T_j\vdash T_{j+1}=T'_{j+1})$. By the inductive assumption we have $(T'_1,\dots,T'_j\vdash T_{j+1}=T'_{j+1})$. Applying (4a) with $\Gamma_1=(T'_1,\dots T'_j)$, $T=T_{j+1}$, $T'=T'_{j+1}$ and $\Gamma_2=(T_{j+2},\dots,T_n)$ we conclude that $(T'_1,\dots,T'_{j+1},T_{j+2},\dots,T_n\vdash S=S')$.

(2)  By the first part of the lemma we have $\Gamma'\vdash S=S'$. Therefore by (4c) it is sufficient to show that $(\Gamma\vdash o=o':S)\wedge(\Gamma\sim\Gamma')\impl (\Gamma'\vdash o=o':S)$. The proof of this fact is similar to the proof of the first part of the lemma using (4b) instead of (4a).  
\end{proof}
\begin{lemma}
\llabel{iseqrelsim}
One has:
\begin{enumerate}
\item Assume that conditions (1.2), (2b), (2c), (2d) and (4a) hold. Then $\sim$ is an equivalence relation.
\item Assume that conditions of the previous part of the lemma as well as conditions (1.3), (3b), (3c), (3d), (4b) and (4c) hold. Then $\simeq$ is an equivalence relation. 
\end{enumerate}
\end{lemma}
\begin{proof}
By induction on $n=l(\Gamma)=l(\Gamma')$. 

(1) Reflexivity follows directly from (1.2) and (2b). For $n=0$ the symmetry is obvious. Let $(\Gamma,T)\sim(\Gamma',T')$. By induction we may assume that $\Gamma'\sim\Gamma$. By Lemma \ref{iseqrelsiml1}(a) we have $(\Gamma'\vdash T=T')$ and by (2c) we have $(\Gamma'\vdash T'=T)$. We conclude that $(\Gamma',T')\sim(\Gamma,T)$.  The proof of transitivity is by a similar induction.

(2) Reflexivity follows directly from  reflexivity of $\sim$, (1.3) and (3b). Symmetry and transitivity are also easy using Lemma \ref{iseqrelsiml1}.
\end{proof}
From this point on we assume that all conditions of Proposition \ref{2014.07.10.prop1}  hold. Let $C'=C/\sim$ and $\wt{C}'=\wt{C}/\simeq$. It follows immediately from our definitions that the functions $ft:C\sr C$ and $\partial:\wt{C}\sr C$ define functions $ft':C'\sr C'$ and $\partial':\wt{C}'\sr C'$.
\begin{lemma}
\llabel{surjl1}
The conditions (3) and (4) of \cite[Proposition 5.4]{Csubsystems} hold for $\sim$ and $\simeq$.
\end{lemma}
\begin{proof}
1. We need to show that for $(\Gamma,T\rhd)$, and $\Gamma\sim\Gamma'$ there exists $(\Gamma',T'\rhd)$ such that $(\Gamma,T)\sim(\Gamma',T')$. It is sufficient to take $T=T'$. Indeed by (2b) we have $\Gamma\vdash T=T$, by Lemma \ref{iseqrelsiml1}(1) we conclude that $\Gamma'\vdash T=T$ and by (1a) that $\Gamma',T\rhd$.  

2.  We need to show that for $(\Gamma\vdash o:S)$ and $(\Gamma,S)\sim(\Gamma',S')$ there exists $(\Gamma'\vdash o':S')$ such that $(\Gamma'\vdash o':S')\simeq(\Gamma\vdash o:S)$. It is sufficient to take $o'=o$. Indeed, by (3b) we have $(\Gamma\vdash o=o:S)$, by Lemma \ref{iseqrelsiml1}(2) we conclude that $(\Gamma'\vdash o=o:S')$ and by (2a) that $(\Gamma'\vdash o:S')$. 
\end{proof}
\begin{lemma}
\llabel{TSetc}
The equivalence relations $\sim$ and $\simeq$ are compatible with the operations $T,\wt{T},S,\wt{S}$ and $\delta$.
\end{lemma}
\begin{proof}
(1) Given $(\Gamma_1,T\rhd)\sim(\Gamma_1',T'\rhd)$ and $(\Gamma_1,\Gamma_2\rhd)\sim(\Gamma_1',\Gamma_2'\rhd)$ we have to show that 
$$(\Gamma_1,T,t_{n+1}\Gamma_2)\sim (\Gamma'_1,T',t_{n+1}\Gamma'_2).$$
where $n=l(\Gamma_1)=l(\Gamma_1')$.

Proceed by induction on $l(\Gamma_2)$. For $l(\Gamma_2)=0$ the assertion is obvious. Let  $(\Gamma_1,T\rhd)\sim(\Gamma_1',T'\rhd)$ and $(\Gamma_1,\Gamma_2,S\rhd)\sim(\Gamma_1',\Gamma_2',S'\rhd)$. The later condition is equivalent to $(\Gamma_1,\Gamma_2\rhd)\sim(\Gamma_1',\Gamma_2'\rhd)$  and $(\Gamma_1,\Gamma_2\vdash S=S')$. By the inductive assumption we have $(\Gamma_1,T,t_{n+1}\Gamma_2)\sim (\Gamma'_1,T',t_{n+1}\Gamma'_2)$. By (5a) we conclude that $(\Gamma_1,T,t_{n+1}\Gamma_2\vdash t_{n+1}S=t_{n+1}S')$. Therefore by definition of $\sim$ we have $(\Gamma_1,T,t_{n+1}\Gamma_2,t_{n+1}S)\sim(\Gamma'_1,T',t_{n+1}\Gamma'_2, t_{n+1}S')$.

(2) Given $(\Gamma_1,T\rhd)\sim(\Gamma_1',T'\rhd)$ and $(\Gamma_1,\Gamma_2\vdash o:S)\simeq(\Gamma_1',\Gamma_2'\vdash o':S')$ we have to show that $(\Gamma_1,T,t_{n+1}\Gamma_2\vdash t_{n+1}o:t_{n+1}S)\simeq (\Gamma'_1,T',t_{n+1}\Gamma'_2\vdash t_{n+1}o':t_{n+1}S')$ where $n=l(\Gamma_1)=l(\Gamma_1')$. We have $(\Gamma_1,\Gamma_2,S)\sim(\Gamma_1',\Gamma'_2,S')$ and $(\Gamma_1,\Gamma_2\vdash o=o':S)$. By (5b) we get $(\Gamma_1,T, t_{n+1}\Gamma_2\vdash t_{n+1}o=t_{n+1}o':t_{n+1}S)$. By (1) of this lemma we get $(\Gamma_1,T,t_{n+1}\Gamma_2,t_{n+1}S)\sim(\Gamma'_1,T',t_{n+1}\Gamma'_2,t_{n+1}S')$ and therefore by definition of $\simeq$ we get $(\Gamma_1,T,t_{n+1}\Gamma_2\vdash t_{n+1}o:t_{n+1}S)\simeq (\Gamma'_1,T',t_{n+1}\Gamma'_2\vdash t_{n+1}o':t_{n+1}S')$.

(3) Given $(\Gamma_1\vdash r:T)\simeq(\Gamma_1'\vdash r':T')$ and $(\Gamma_1,T,\Gamma_2\rhd)\sim(\Gamma_1',T',\Gamma_2'\rhd)$ we have to show that 
$$(\Gamma_1,s_{n+1}(\Gamma_2[r/n+1]))\sim(\Gamma'_1,s_{n+1}(\Gamma'_2[r'/n+1])).$$
where $n=l(\Gamma_1)=l(\Gamma_1')$. Proceed by induction on $l(\Gamma_2)$. For $l(\Gamma_2)=0$ the assertion follows directly from the definitions. Let $(\Gamma_1\vdash r:T)\simeq(\Gamma_1'\vdash r':T')$ and $(\Gamma_1,T,\Gamma_2,S\rhd)\sim(\Gamma_1',T',\Gamma_2',S'\rhd)$. The later condition is equivalent to $(\Gamma_1,T,\Gamma_2\rhd)\sim(\Gamma_1',T',\Gamma_2'\rhd)$  and $(\Gamma_1,T,\Gamma_2\vdash S=S')$. By the inductive assumption we have $(\Gamma_1,s_{n+1}(\Gamma_2[r/n+1]))\sim(\Gamma'_1,s_{n+1}(\Gamma'_2[r'/n+1]))$. It remains to show that $(\Gamma_1,s_{n+1}(\Gamma_2[r/n+1])\vdash s_{n+1}(S[r/n+1])=s_{n+1}(S'[r'/n+1]))$. By (2d) it is sufficient to show that $(\Gamma_1,s_{n+1}(\Gamma_2[r/n+1])\vdash s_{n+1}(S[r/n+1])=s_{n+1}(S'[r/n+1]))$ and $(\Gamma_1,s_{n+1}(\Gamma_2[r/n+1])\vdash s_{n+1}(S'[r/n+1])=s_{n+1}(S'[r'/n+1]))$. The first relation follows directly from (6a). To prove the second one it is sufficient by (7a) to show that $(\Gamma_1,T,\Gamma_2,S'\rhd)$ which follows from our assumption through (2c) and (2a). 

(4) Given $(\Gamma_1\vdash r:T)\simeq(\Gamma_1'\vdash r':T')$ and $(\Gamma_1,T,\Gamma_2\vdash o:S)\simeq(\Gamma_1',T',\Gamma_2'\vdash o':S')$ we have to show that 
$$(\Gamma_1,s_{n+1}(\Gamma_2[r/n+1])\vdash s_{n+1}(o[r/n+1]):s_{n+1}(S[r/n+1]))\simeq$$ 
$$ (\Gamma'_1,s_{n+1}(\Gamma'_2[r'/n+1])\vdash s_{n+1}(o'[r'/n+1]):s_{n+1}(S'[r'/n+1])).$$
where $n=l(\Gamma_1)=l(\Gamma_1')$ or equivalently that 
$$(\Gamma_1,s_{n+1}(\Gamma_2[r/n+1]),s_{n+1}(S[r/n+1]))\sim(\Gamma'_1,s_{n+1}(\Gamma'_2[r'/n+1]), s_{n+1}(S'[r'/n+1]))$$
and $(\Gamma_1,s_{n+1}(\Gamma_2[r/n+1])\vdash s_{n+1}(o[r/n+1])=s_{n+1}(o'[r'/n+1]):s_{n+1}(S[r/n+1]))$. The first statement follows from part (3) of the lemma. To prove the second statement it is sufficient by (3d) to show that  $(\Gamma_1,s_{n+1}(\Gamma_2[r/n+1])\vdash s_{n+1}(o[r/n+1])=s_{n+1}(o'[r/n+1]):s_{n+1}(S[r/n+1]))$ and  $(\Gamma_1,s_{n+1}(\Gamma_2[r/n+1])\vdash s_{n+1}(o'[r/n+1])=s_{n+1}(o'[r'/n+1]):s_{n+1}(S[r/n+1]))$. The first assertion follows directly from (6b). To prove the second one it is sufficient in view of (7b) to show that $(\Gamma_1,T,\Gamma_2\vdash o':S)$ which follows conditions (3c) and (3a).

(5) Given $(\Gamma,T)\sim(\Gamma',T')$ we need to show that $(\Gamma,T\vdash (n+1):T)\simeq(\Gamma',T'\vdash (n+1):T')$ or equivalently that $(\Gamma,T,T)\sim(\Gamma,T',T')$ and $(\Gamma,T\vdash (n+1)=(n+1):T)$. The second part follows from (3b). To prove the first part we need to show that $(\Gamma,T\vdash T=T')$. This follows from our assumption by (5a). 
\end{proof}
\begin{lemma}
\llabel{2014.07.12.l1}
Let $C$ be a subset of $Ob(CC(R,LM))$ which is closed under $ft$. Let $\le$ be a transitive relation on $C$ such that:
\begin{enumerate}
\item $\Gamma\le \Gamma'$ implies $l(\Gamma)=l(\Gamma')$,
\item $\Gamma\in C$ and $ft(\Gamma)\le F$ implies $\sigma(\Gamma,F)\in C$ and $\Gamma\le \sigma(\Gamma,F)$.
\end{enumerate}
Then $\Gamma\in C$ and $ft^i(\Gamma)\le F$ for some $i\ge 1$, implies that $\Gamma\le \sigma(\Gamma,F)$. 
\end{lemma}
\begin{proof}
Simple induction on $i$.
\end{proof}
\begin{lemma}
\llabel{2014.07.12.l2}
Let $C$ and $\le$ be as in Lemma \ref{2014.07.12.l1}. Then one has:
\begin{enumerate}
\item $(\Gamma,T)\le (\Gamma,T')$ and $\Gamma\le \Gamma'$ implies that $(\Gamma,T)\le (\Gamma',T')$,
\item if $\le$ is $ft$-monotone (i.e. $\Gamma\le \Gamma'$ implies $ft(\Gamma)\le ft(\Gamma')$) and symmetric then $(\Gamma,T)\le (\Gamma',T')$ implies that $(\Gamma,T)\le (\Gamma,T')$.
\end{enumerate}
\end{lemma}
\begin{proof}
The first assertion follows from
$$(\Gamma,T)\le (\Gamma,T')\le \sigma((\Gamma,T'),\Gamma')=(\Gamma',T')$$
The second assertion  follows from
$$(\Gamma,T)\le (\Gamma',T')\le \sigma((\Gamma',T'),\Gamma)=(\Gamma,T')$$
where the second $\le$ requires $\Gamma'\le \Gamma$ which follows from $ft$-monotonicity and symmetry.
\end{proof}
\begin{lemma}
\llabel{2014.07.12.l3}
Let $C,\le$ be as in Lemma \ref{2014.07.12.l1}, let $\wt{C}$ be a subset of $\wt{Ob}(CC(R,LM))$ and $\le'$ a transitive relation on $\wt{C}$ such that: 
\begin{enumerate}
\item ${\cal J}\le' {\cal J}'$ implies $\partial({\cal J})\le\partial({\cal J}')$,
\item ${\cal J}\in \wt{C}$ and $\partial({\cal J})\le F$ implies $\wt{\sigma}({\cal J},F)\in \wt{C}$ and ${\cal J}\le' \wt{\sigma}({\cal J},F)$.
\end{enumerate}
Then ${\cal J}\in \wt{C}$ and $ft^i(\partial({\cal J}))\le F$ for some $i\ge 0$ implies ${\cal J}\le \wt{\sigma}({\cal J},F)$. 
\end{lemma}
\begin{proof}
Simple induction on $i$.
\end{proof}
\begin{lemma}
\llabel{2014.07.12.l4}
Let $C,\le$ and $\wt{C},\le'$ be as in Lemma \ref{2014.07.12.l3}. Then one has:
\begin{enumerate}
\item $(\Gamma\vdash o:T)\le' (\Gamma\vdash o':T)$ and $(\Gamma,T)\le (\Gamma',T')$ implies that $(\Gamma\vdash o:T)\le' (\Gamma'\vdash o':T')$,
\item if $(\le,\le')$ is $\partial$-monotone (i.e. ${\cal J}\le' {\cal J}'$ implies $\partial({\cal J})\le \partial({\cal J}')$) and $\le$ is symmetric then $(\Gamma\vdash o:T)\le' (\Gamma'\vdash o':T')$ implies that $(\Gamma\vdash o:T)\le' (\Gamma\vdash o':T)$.
\end{enumerate}
\end{lemma}
\begin{proof}
The first assertion follows from
$$(\Gamma\vdash o:T)\le'  (\Gamma\vdash o':T)\le' \wt{\sigma}((\Gamma\vdash o':T) ,(\Gamma',T'))=(\Gamma'\vdash o':T')$$
The second assertion follows from
$$\Gamma\vdash o:T)\le' (\Gamma'\vdash o':T')\le' \sigma((\Gamma'\vdash o':T'),(\Gamma,T))=(\Gamma\vdash o':T)$$
where the second $\le$ requires $\Gamma'\le \Gamma$ which follows from $\partial$-monotonicity of $\le'$ and symmetry of $\le$.
\end{proof}

\begin{proposition}
\llabel{2014.07.10.prop2}
Let $(C,\wt{C})$ be subsets in $Ob(CC(R,LM))$ and $\wt{Ob}(CC(R,LM))$ respectively which correspond to a C-subsystem $CC$ of $CC(R,LM)$. Then the constructions presented above establish a bijection between pairs of subsets $(Ceq,\wt{Ceq})$ which together with $(C,\wt{C})$ satisfy the conditions of Proposition \ref{2014.07.10.prop1} and pairs of equivalence relations $(\sim,\simeq)$ on $(C,\wt{C})$ such that:
\begin{enumerate}
\item $(\sim,\simeq)$ corresponds to a regular congruence relation on $CC$ (i.e., satisfies the conditions of \cite[Proposition 5.4]{Csubsystems}),
\item $\Gamma\in C$ and $ft(\Gamma)\sim F$ implies $\Gamma\sim \sigma(\Gamma,F)$,
\item ${\cal J}\in \wt{C}$ and $\partial({\cal J})\sim F$ implies ${\cal J}\simeq \wt{\sigma}({\cal J},F)$.
\end{enumerate}
\end{proposition}
\begin{proof}
One constructs a pair $(\sim,\simeq)$ from $(Ceq,\wt{Ceq})$ as in Definition \ref{simandsimeq}. 
This pair corresponds to a regular congruence relation by Proposition \ref{2014.07.10.prop1}.
Conditions (2),(3) follow from Lemma \ref{iseqrelsiml1}.

Let $(\sim,\simeq)$ be equivalence relations satisfying the conditions of the proposition. Define $Ceq$ as the set of sequences $(\Gamma,T,T')$ such that $(\Gamma,T), (\Gamma,T')\in C$ and $(\Gamma,T)\sim (\Gamma,T')$. Define $\wt{Ceq}$ as the set of sequences $(\Gamma,T,o,o')$ such that $(\Gamma,T,o),(\Gamma,T,o')\in \wt{C}$ and $(\Gamma,T,o)\simeq (\Gamma,T,o')$. 

Let us show that these subsets satisfy the conditions of Proposition \ref{2014.07.10.prop1}. Conditions (2.a-2.d) and (3.a-3d) are obvious. 

Condition (4a) follows from (2) by Lemma \ref{2014.07.12.l1}.
Conditions (4b) and (4c) follow from (3) by Lemma \ref{2014.07.12.l3}.

Conditions (5a) and (5b) follow from the compatibility of $(\sim,\simeq)$ with $T$ and $\wt{T}$. 

Conditions (6a),(6b),(7a),(7b) follow from the compatibility of $(\sim,\simeq)$ with $S$ and $\wt{S}$.
\end{proof}

\subsection{Pairs $(R,LM)$ associated with nominal signatures.}
\llabel{2014.07.22.sec}
The constructions of this paper produce C-systems from a pair $(R,LM)$ where $R$ is a monad on $Sets$ and $LM$ is a left $R$-module with values in $Sets$ together with sets $C$, $\wt{C}$, $Ceq$ and $\wt{Ceq}$. 

One class of such pairs is obtained by taking $R$ to be the monad defined by a signature as in \cite[p.228]{HM2007}. For example, the contextual category of the Martin-Lof type theory from 1972,  $MLTT72$ defined in \cite{ML72}, is obtained by applying Proposition \ref{2014.07.10.prop1} in the case of the pair $(R,R)$ where $R$ is the monad defined by the signature that corresponds to the nominal signature of Example \ref{2014.08.ex}.

The following construction that covers more examples associates a pair $(R,LM)$ to a quadruple $(\Sigma, Term, P, {\bf Type})$ where $\Sigma$ is a nominal signature with one name-sort $Var$ and a set of data-sorts $D$, $Term\in {\bf D}$ is a data-sort, $P$ is a family of sets parametrized by ${\bf D}-\{Term\}$,  and ${\bf Type}\subset \{\bf D\}$ is a subset of data-sorts. 

In most examples either ${\bf D}=\{Term\}$ or ${\bf D}=\{Term,Type\}$, ${\bf Type}=\{Type\}$ and $P=P_{Type}$ is the set of "type-variables". 

The only example which I know of where there are more than two data-sorts is the logic-enriched type theory of \cite{AczelGambino} where ${\bf D}=\{Term, Type, Prop\}$, ${\bf Type}=\{Type\}$, $P_{Type}$ is the set of type variables and $P_{Prop}$ is the set of propositional variables. 

The construction is as follows.  A {\em nominal signature} (see \cite[Section 8.1]{Pitts}) starts with a set of name-sorts $\bf N$ and the set of data-sorts $\bf D$. We will be interested in the case when there is only one name-sort $Var$.

A compound sort $S$ is defined as an expression formed from $Var$, elements of $\bf D$, and the unit sort $1$ using two operations: one sending $S_1$ and $S_2$ to $(S_1,S_2)$ and another one sending $S$ to $Var.S$.  For better notations one takes $(\_,\_)$ to associate on the left i.e. $(S_1,S_2,S_3)$ means $((S_1,S_2),S_3)$ and similarly for longer sequences.

 Let $CS$ be the set of compound sorts. An arity is a pair $(S,D)$ where $S\in CS$ and $D\in {\bf D}$. Let $A({\bf D})$ be the set of arities for the set of data-sorts $\bf D$. 

A nominal signature is defined as a set $Op$, which is called the set of operations, together with a function  $Ar:Op\sr A({\bf D})$ which assigns to any operation its ``arirty". One writes $O:S\sr D$ to denote that operation $O$ has arity $(S,D)$. We let $Ar_{CS}$ and $Ar_{\bf D}$ denote the two components of the arity. 

For example, the nominal signature of the lambda calculus has one data-sort $Term$ and three operations $V$, $L$, and $A$  of the form:
$$V:Var\sr Term$$
$$L:Var.Term\sr Term$$
$$A:Term.Term\sr Term$$
The algebraic signature with one sort $Term$, one operation $S$ in one variable and one constant $O$ will, in this language, have {\em three}  operations:
$$V:Var\sr Term$$
$$S: Term\sr Term$$
$$O:1\sr Term$$
More generally, an algebraic signature is a nominal signature where 
$$Op=Op_0\coprod\{v_D\}_{D\in {\bf D}}$$
with 
$$v_D:Var\sr D$$
and for $O\in Op_0$
$$O:(D_1,\dots,D_n)\sr D$$
for some $n\ge 0$ and $D_1,\dots,D_n,D\in {\bf D}$ where $n$ and $D$'s may depend on $O$.  

An example of a signature where variables are not terms is given in \cite{Pitts}. 

A nominal signature can be used to construct terms of all compound sorts in the more or less obvious way. Next one defines the notion free and bound occurrences of variables in these terms and the notion of the  $\alpha$-equivalence. For a nominal signature $\Sigma$ and a compound sort $S$ one writes $\Sigma_{\alpha}(S)$ for the set of $\alpha$-equivalence classes of terms of sort $S$ build using $\Sigma$.

In the case when $\Sigma$ is the $\lambda$-calculus signature one gets the usual set of $\alpha$-equivalence classes of $\lambda$-terms considering $\Sigma_{\alpha}(Term)$. 

To any nominal signature $\Sigma$ one associates, following \cite{Pitts}, a functor $T_{\Sigma}:Nom^{\bf D}\sr Nom^{\bf D}$ where $Nom$ is the category of nominal sets, as follows.

First one associates a functor $[S]:Nom^{\bf D}\sr Nom$ to any compound sort $S$ by the rule:
$$[Var](X)={\bf A}$$
$$[D](X)=X_D$$
$$[1](X)=1$$
$$[(S_1,S_2)](X)=X\times X$$
$$[(Var.S)]=[{\bf A}](X)$$
where ${\bf A}$ is the standard atomic nominal set (the set of names with the canonical action of the permutation group $Perm$) and $[{\bf A}]$ is the name-abstraction functor $Nom\sr Nom$ which is defined in \cite[Section 4]{Pitts}. 

Let $Op_D$ for $D\in {\bf D}$ be the set of operations $O$ with the target sort $D$, i.e., such that $Ar_{\bf D}(O)=D$.  Then one defines $T_{\Sigma}(X)$ by the rule
$$T_{\Sigma}(X)_D=\coprod_{O\in Op_D} [Ar_{CS}(O)].$$
For example, if $\Sigma$ is the signature of $\lambda$-calculus then 
$$T_{\Sigma}(X)={\bf A}\coprod [{\bf A}](X)\coprod (X\times X)$$

One of the main results of \cite{Pitts} is that the functor $T_{\Sigma}$ has an initial algebra $I_{\Sigma}$ for any $\Sigma$ and $(I_{\Sigma})_D=\Sigma_{\alpha}(D)$. 

Let us extend this construction to a monad on $Nom^{\bf D}$ and then on $Sets^{\bf D}$. First observe that for any $X\in Nom^{\bf D}$ the functor $Y\mapsto T_{\Sigma}(Y)\coprod X$ is finitely presented and therefore it has an initial algebra. Let us denote this algebra by $NR_{\Sigma}(X)$. 

By \cite[pp. 243-244]{Awodey2010}, $NR_{\Sigma}$ is a monad on $Nom^{\bf D}$ whose category of algebras is equivalent to the category of $T_{\Sigma}$-algebras (i.e. $NR_{\Sigma}$ is the free monad generated by $T_{\Sigma}$).

The functor $Discr:Sets \sr Nom$ which takes a set to the corresponding discrete nominal set has a right adjoint $Inv:Nom\sr Sets$ which sends a nominal set $X$ to the set of its fixed points $X^{Perm}$. The functors $Discr^{\bf D}$ and $Inv^{\bf D}$ form an adjoint pair between the categories $Nom^{\bf D}$ and $Sets^{\bf D}$.

Given a monad $R$ on a category $\cal C$ and an adjoint pair $(LF,RF)$ where $RF:{\cal C}\sr{\cal C}'$ is the right adjoint, the composition $R'=RF\circ R\circ LF$ is a monad on ${\cal C}'$.

Applying this fact to the monad $NR_{\Sigma}$ and the pair $(Discr^{\bf D},Inv^{\bf D})$ we conclude that the functor 
$$R_{\Sigma}:X\mapsto NR_{\Sigma}(Discr^{\bf D}(X))^{Perm}$$
is a monad on $Sets^{\bf D}$.

For a family of sets $X$ the functor $T_{\Sigma}\coprod Discr^{\bf D}(X)$ is naturally isomorphic to the functor $T_{\Sigma+X}$ where $\Sigma+X$ is the signature with the set of operations $Op\coprod (\coprod_{D\in {\bf D}} X_D)$ and the arity function defined on $x\in X_D$ by $Ar(x)=(1,D)$ and
$$R_{\Sigma}(X)=NR_{\Sigma}(Discr^{\bf D}(X))^{Perm}=I_{T_{\Sigma}\coprod Discr^{\bf D}(X)}^{Perm}=I_{T_{\Sigma+X}}^{Perm}.$$
Therefore
$$(R_{\Sigma}(X))_D=(\Sigma+X)_{\alpha}(D)^{Perm}$$
is the set of invariants in the set of $\alpha$-equivalence classes of terms of sort $D$ with respect to the signature $\Sigma+X$ i.e. the set of $\alpha$-equivalence classes of closed terms of sort $D$ with respect to $\Sigma+X$. 

If $X_D=\{x_{1,D},\dots,x_{n_D,D}\}$ are finite sets, then the terms with respect to the signature $\Sigma+X$ can be seen as terms with respect to $\Sigma$ which depend on additional parameters $x_{i,D}$ of the corresponding sorts and the closed terms as the terms with respect to $\Sigma$ relative to the name space ${\bf A}^{\bf D}+X$ such that all the occurrences of names from ${\bf A}^{\bf D}$ are bound and all the occurrences of names from $X$ are free. 

To obtain from this construction a pair $(R,LM)$ of a monad on $Sets$ and a left module over this monad with values in $Sets$ we will use Lemma \ref{2014.07.28.l3}. Let $Term\in{\bf D}$ and ${\bf Type}\subset{\bf D}$. Let $P$ a family of sets parametrized by ${\bf D}-{Term}$. For a set $X$ let $(X,P)$ be the family such that $(X,P)_{Term}=X$ and $(X,P)_{D}=P_{D}$ for $D\ne Term$. 

Then $X\mapsto (R_{\Sigma}(X,P))_{Term}$ is a monad $R_{\Sigma,Term,P}$ on $Sets$ by Lemma \ref{2014.07.28.l2} and 
$$X\mapsto \coprod_{D\in {\bf Type}} (R_{\Sigma}(X,P))_D$$
is a left module $LM_{\Sigma,Term,P,{\bf Type}}$ over $R_{\Sigma,Term,P}$ by Lemmas \ref{2014.07.28.l3} and \ref{2014.07.28.l1}(b). 
\begin{example}\rm
The C-systems of generalized algebraic theories (GATs) of \cite{Cartmell0},\cite{Cartmell1} (see also \cite{Garner}) are obtained by using algebraic signatures with two data sorts ${\bf D}=\{Term,Type\}$, ${\bf Type}=\{Type\}$ and $P=\emptyset$. The ``symbols'' of the GAT are operations of the corresponding algebraic signature. The term symbols of degree $n$ have arity $(Term,\dots,Term)\sr Term$ and the type symbols of degree $n$ have arity $(Term,\dots,Term)\sr Type$ where in both cases the lentth of the sequence $(Term,\dots,Term)$ is $n$. 
\end{example}
\begin{example}
\llabel{2014.08.ex}\rm
To define the Martin-Lof Type theory MLTT72 of \cite{ML72}  one needs to consider the case when ${\bf D}=\{Term\}$ and the nominal signature is of the form:
$$v:Var\sr Term$$
$$\Pi:(Term, Var.Term)\sr Term\spc\lambda:(Term,Var.Term)\sr Term$$ $$app:(Term,Term)\sr Term$$
$$\Sigma:(Term,Var.Term)\sr Term\spc pair:(Term,Term)\sr Term$$ $$E:(Term,Var.(Var.Term))\sr Term$$
$$+:(Term,Term)\sr Term\spc i:Term\sr Term\spc j:Term\sr Term$$ $$D:((Term,Var.Term),Var.Term))\sr Term$$
$$V:1\sr Term$$
$$N_n:1\sr Term\spc i_n:1\sr Term\spc R_n:(Term,\dots,Term),\dots)\sr Term$$ $$n\ge 0\spc 1\le i \le n$$
$$N:1\sr Term\spc 0:1\sr Term\spc s:Term\sr Term$$ $$R:((Term,Term),(Var.(Var.Term)))\sr Term$$

Note that in fact $E$, $D$, $R_n$, and $R$ should also have the type family $C$ (see \cite[2.3.6, 2.3.8, 2.3.10, 2.3.12]{ML72}) as an argument which, in our notation, means an additional component of the form $Var.Term$ in their arities.  

In fact, the original definition from \cite{ML72} allows for additional ``type constants'' (see \cite[2.2.1]{ML72}) of various algebraic arities which are analogous to the predicate constants in the predicate logic. As such it is a definition of a family of type systems. The signatures underlying all type systems in this family are obtained by extending the signature described above by a set of operations of the form $P:(Term,\dots,Term)\sr Term$. 

For the signature of the MLTT78 see \cite[p. 158]{ML78}
\end{example}
\begin{remark}\rm
It is possible to ``encode'' a nominal signature in typed $\lambda$-calculus using the idea that closed terms are objects of a base type $term$, terms with one free variable are objects of the type $term\sr term$, terms with two free variables are objects of $term\sr term\sr term$ etc. This encoding allows one to describe the substitutions of closed terms into terms with free variables as applications in the meta-theory.  However, it does not allow to describe the substitution of, e.g., terms with one free variable into terms with one free variable, i.e., the full monadic structure is not recoverable from such a description. This is the reason why the use of typed $\lambda$-calculus systems such as the Logical Framework for the description of the syntax of dependent type theories is of limited use.
\end{remark}

\def\cprime{$'$}

\end{document}